\documentclass[12pt, a4paper]{article} 
\usepackage[english]{babel}
\usepackage{amsmath,amssymb,amsfonts,amscd}
\usepackage{amsthm}
\usepackage[T1]{fontenc}
\usepackage{latexsym}
\usepackage{amssymb,amsbsy,amsmath,amscd}
\usepackage{amsmath,amssymb,amsbsy, amsmath, amsfonts, amscd}
\usepackage[mathcal]{eucal}
\usepackage[pdftex]{color,graphicx}
\usepackage{latexsym}
\usepackage{fancyhdr}
\usepackage{cite}
\newcommand{\cc}{\mathcal}
\newcommand{\rr}{\mathbb{R}}
\newcommand{\bb}{\mathbb}
\newcommand{\rn}{\mathbb{R}^N}
\newcommand{\sla}{\backslash}

\newcommand{\bdisp}{\begin{displaymath}}
\newcommand{\edisp}{\begin{displaymath}}
\newcommand{\bsplit}{\begin{split}}
\newcommand{\esplit}{\end{split}}
\newcommand{\bacc}{\left\{}
\newcommand{\eacc}{\right\}}
\newcommand{\bp}{\left(}
\newcommand{\ep}{\right)}
\newcommand{\bint}{\left[}
\newcommand{\eint}{\right]}
\newcommand{\bnorme}{\left\|}
\newcommand{\enorme}{\right\|}
\newcommand{\babsolu}{\left|}
\newcommand{\eabsolu}{\right|}

\newcommand{\be}{\begin{equation}}
\newcommand{\ee}{\end{equation}}
\newcommand{\ba}{\begin{array}}
\newcommand{\ea}{\end{array}}
\newcommand{\bdes}{\begin{description}}
\newcommand{\edes}{\end{description}}
\newcommand{\benu}{\begin{enumerate}}
\newcommand{\eenu}{\end{enumerate}}

\newcommand{\associe}{\mapsto}
\newcommand{\avaleur}{\rightarrow}
\newtheorem{defi}{Definition}[section]
\newtheorem{lem}[defi]{Lemma}
\newtheorem{pro}[defi]{Proposition}
\newtheorem{thm}[defi]{Theorem}
\newtheorem{cor}[defi]{Corollary}
\newtheorem{rmq}[defi]{Remark}
\newtheorem{exemple}[defi]{Example}
\newcommand{\dal}{\Delta^{\frac{\alpha}{2}}}
\newcommand{\daln}{(-\Delta)^{\frac{\alpha}{2}}}
\newcommand{\hd}{H_D^{\alpha,\varphi}}
\title{Nonnegative entire bounded solutions to some semilinear equations involving the fractional Laplacian}
\author{Mohamed Ben Chrouda$^a$ \and Mahmoud Ben Fredj$^b$}
   
\begin{document}
\maketitle
\begin{center}{\scriptsize{\emph{$^a$Dep. of Mathematics, High institute of computer sciences and mathematics, 5000, Monastir, Tunisia }}}\end{center}
\begin{center}{\scriptsize{\emph{$^b$Dep. of Mathematics, Faculty of Sciences of Monastir, 5019 Monastir, Tunisia}}}\end{center}
\begin{abstract} The main goal  is to establish necessary and sufficient conditions  under which the fractional semilinear elliptic equation $\dal u=\rho(x)\,\varphi(u)$  admits  nonnegative nontrivial  bounded solutions in the whole space $\rn$. \end{abstract}
\section{Introduction}
Let $\alpha\in]0,2[,$ $p\in\rr$ and let  $\rho$ be a nonnegative locally bounded Borel function in $\rn,$ $N\geq 3.$   Our main goal in this paper is to derive  sufficient and necessary conditions for existence of  nontrivial  bounded solutions to the equation
\[\dal u=\pm \rho\,u^p \]in the whole space $\rn,$ where $\dal$ stands for the fractional Laplacian $-(-\Delta)^{\frac{\alpha}{2}}$ which is the infinitesimal generator of the standard symmetric $\alpha$-stable process in $\rn$ and which appears, among other fields, in anomalous diffusions in plasma, flames propagation and chemical reactions in liquids.

In order to reach our purpose, we shall study, in two different situations (which are specified below), the more general equation
\begin{equation}\label{ei2}\dal u= \pm\rho\,\varphi(u)\end{equation} in the whole space $\rn,$ where $\varphi$ is a nonnegative real-valued Borel function. Solutions of such equations are understood in the distributional sense and are called entire solutions in the literature. Along the paper we will look only for nonnegative solutions, so many times we will omit the term "nonnegative".

For the limiting case $\alpha=2,$ equations of the kind of (\ref{ei2}) have been the main subject of investigation in a large amount of works. Various hypothesis on $\rho$ and on the nonlinearity $\varphi$ have been considered. Without any attempt to review the references  here,  one can see \cite{mahmoud,bresiz,delPino,edelson,8,khalifabounded,kusano,lair,lairwood,lazer} (the list is far from complete).
Such equations have been investigated in different classes of domains, bounded and unbounded,  with several kinds of smoothness.

Recently, several studies have been performed for classical elliptic equations with the Laplacian operator substituted by its fractional powers \cite{autuori,barrios,birkner,brandle,maagli,fall,felmerquaas,felmerwang,servadei} by using almostly variational and partial differential equations's related techniques. In particular, there has been an interest to the  solutions in the whole space $\rn$ (see, for instance, \cite{autuori,birkner,felmerquaas}). 
  In this direction, we intend to make some contribution towards the existing literature by studying two kinds of semilinear equations.  Both of them are one of the most commonly considered in the literature for the classical Laplacian. However, to the best of our knowledge, such equations (in the whole space $\rn$) have not yet been studied  in the framework of the fractional Laplacian.
  
  In the first situation, we are interested in the following nonlinear equation
  \begin{equation}\label{eip}\dal u= \rho\,\varphi(u)\end{equation}
  where $\varphi :[0,+\infty[\to [0,+\infty[  $ is continuous and nondecreasing such that $\varphi(0)=0$ (the reference example corresponds to the case $\varphi(t)=t^p,$ $p>0$).
   We prove that Eq. (\ref{eip}) admits a nontrivial  entire  bounded solution if and only if there exists a transient  set $A$ (we recall the definition below) and $x_0\in\rn$ such that 
 \[\int_{A^c}\frac{\rho(y)}{\babsolu x_0-y\eabsolu^{N-\alpha}}dy<\infty\]
 where $A^c$ is the complementary of the set $A$ in $\rn.$
 It is not hard to see that  the above condition holds true (with $A=\emptyset$ and $x_0=0$) whenever 
 \begin{equation}\label{a1}\int_\eta^\infty r^{\alpha-1}\rho^*(r)\,dr<\infty,\end{equation}
for some $\eta>0,$ where $\rho^*$ is defined for every $r\geq0$ by $\rho^*(r)=\sup_{\babsolu x\eabsolu=r}\rho(x).$ We shall prove that the converse is also true in the case where $\rho$ is radially symmetric on $\rn.$ However, for a genaral $\rho$ we verify that (\ref{a1}) is not a necessary condition. For the purpose hereof, we shall prove that every classical superharmonic function is $\alpha$-superharmonic (relatively to the fractional Laplacian), a result which turns out to be of interest in itself .

In the second place, we focus on the following negative singular perturbation 
 \begin{equation}\label{ein}\daln u= \rho\,\varphi(u)\end{equation}
  where $\varphi :]0,+\infty[\to [0,+\infty[  $ is continuous and nonincreasing (the reference example corresponds to the case $\varphi(t)=t^p,$ $p<0$).
 We prove that Eq. (\ref{ein}) admits a nontrivial  entire  bounded solution if and only if $\rho$ is potentially bounded, that is,
\begin{equation}\label{aa}\sup_{x\in\rn}\int_{\rn}\frac{\rho(y)}{\babsolu x-y\eabsolu^{N-\alpha}}dy<\infty.\end{equation}
In the particular case where $\rho$ is radially symmetric, we shall show that (\ref{aa}) is equivalent to the condition 
\begin{equation}\label{aaa}\int_\eta^\infty r^{\alpha-1}\rho(r)\,dr<\infty\end{equation}
for some $\eta>0.$
Furthermore, we prove  that Eq. (\ref{ein}) admits a nonnegative entire bounded solution decaying to zero at infinity whenever (\ref{aa}) or (\ref{aaa}) holds true.

  Considerable progress has been made recently in extending potential-theoretic proprieties of Brownian motion to symmetric $\alpha$-stable process (see for example \cite{bet,bogdanliouville,bogdan1997,bogdan1999,bogdanby,bogdan2009,chengauge,chensong1997,chensong1998,michalik}). In this paper, sometimes we have cited some references dealing with Brownian motion but the needed proofs and technics works systimaticaly for any regular Markov process and in particular for the $\alpha$-stable symmetric process.

An important feature of the fractional Laplacian is its nonlocal property, which makes it difficult to handle. Needless to say, for the classical Dirichlet problem, the boundary datum is concentrated only on $\partial D$ whereas the Dirichlet condition for fractional Laplacian must prescribed in all $D^c.$ Another simple observation at this early stage is that in the classical setting, the maximum principle states that a subharmonic function is bounded above inside a domain by its values on the boundary. As a consequence, every radial subharmonic function $s$ in $\rn$ is increasing in the sense that $s(x)\geq s(x_0)$ for every $\babsolu x\eabsolu\geq \babsolu x_0\eabsolu.$ This property is no longer true (in general) for $\alpha$-subharmonic functions since the maximum needs not to be reached at the boundary but eventually at a exterior point.

Our results follow  up those obtained in \cite{khalifabounded,8} for the classical Laplacian. Our development is a standard one. It combines probabilistic and analytic tools from potential theory. However, most of the arguments are susbtancially modified in comparison with its classical counterparts. 

\section{Preliminaries}

 For every subset  $F$ of $\bb{R}^N$, let $\cc{B}(F)$ be the set of all Borel measurable functions on $F$ and
 let $\cc{C}(F)$ be the set of all continuous real-valued functions on $F.$ If $\cc{G}$ is a set of numerical
  functions then $\cc{G}^+$ (respectively $\cc{G}_b$) will denote the class of all functions in $\cc{G}$ which
  are nonnegative (respectively bounded). $\cc{C}^k(F)$ is the class of all functions that are $k$ times   continuously differentiable on $F$ and $\cc{C}_0(F)$ is the set of all continuous  functions on $F$ such that $u=0$ on $\partial F,$ which means that $\lim_{x\to z}u(x)=0$ for all $z\in\partial F$ and $\lim_{x\to\infty}u(x)=0$ if $F$ is unbounded.  The uniform  norm will be denoted by $\left\|\cdot\right\|.$

Let $\alpha\in ]0,2[$ and $N\geq3.$ We denote by $(\Omega,X_t,P^x)$ the standard rotation (symmetric) invariant stable process in $\rn,$ with index of stability $\alpha,$ and  characteristic function 
$$E^0[e^{i<\xi,X_t>}]=\int_{\rn} e^{i<\xi,x>}p(t,x)\,dx=e^{-t\babsolu \xi\eabsolu^\alpha}; \quad \xi\in \rn,\ t\geq0,$$
where $p(t,x,y)=p(t,x-y)$ is the transition density of the process which is uniquely determined by its Fourier tranform.
As usual, $E^x$ is the expectation with respect to the distribution $P^x$ of the process starting from $x\in\rn.$ The limiting classical case $\alpha=2$ corresponds to the Brownian motion with Laplacian $\Delta=\sum_{i=1}^N\partial_i^2$ as generator. Nevertheless, when $0<\alpha<2,$ the process has as  generator the fractional Laplacian $\dal$  which  is a prototype of non-local operators.

For the reader's convinence, we recall the definition of the fractional Laplacian. We denote by $\mathcal{L}_\alpha$ the set of all Borel measurable functions $u:\rn\to \rr$ such that 
\[\int_{\rn}\frac{\babsolu u(y)\eabsolu}{\bp1+\babsolu y\eabsolu\ep^{N+\alpha}}dx<\infty.\] 
Note that  bounded Borel measurable functions are in  $\mathcal{L}_\alpha.$ The fractional power of the Laplacian $\dal $ is defined by 
$$\Delta^{\frac{\alpha}{2}}u(x)=c_{N,-\alpha}\int_{\rn}\frac{u(y)-u(x)}{\babsolu y-x\eabsolu^{N+\alpha}}\,dy\quad;\qquad x\in\rn,$$
for every Borel function $u$ for which the integral exists.
The constant $c_{N,-\alpha}$ is  depending only on $N$ and $\alpha:$ $c_{N,-\alpha}=2^{\alpha}\pi^{-\frac{N}{2}}\babsolu\Gamma(\frac{-\alpha}{2})\eabsolu^{-1}\Gamma(\frac{N+\alpha}{2}).$  We point out   that $\dal u$ is well defined for every $u\in \mathcal{L}_\alpha\cap \mathcal{C}^2.$ However,   we can define $\dal $ as a distribution in $\mathcal{L}_\alpha$ by 
\[<\dal u,\theta>=\int_{\rn}u(y)\,\dal\theta(y)\,dy\quad;\quad \theta\in \mathcal{C}_c^\infty(\rn),\]
where $ \mathcal{C}_c^\infty(\rn)$ is the set of all infinitely differentiable functions on $\rn$ with compact support.
 
  Let $D$ be a bounded domain in $\rn$ and let   $\tau_D$ be  the first exit time from $D$ by $X,$
 i.e.,
$$
\tau_D=\inf\left\{t>0;X_t\notin D\right\}.
$$ 
Let $u$ be a Borel measurable locally integrable function on $\rn.$ We say that $u$ is $\alpha$-harmonic in $D$  if 
\begin{equation}\label{harmonic}
E^x\bint\babsolu u(X_{\tau_U})\eabsolu\eint<\infty \quad\mbox{and}\quad u(x)=E^x[u(X_{\tau_U})]\quad;\quad x\in U,
\end{equation} 
for every bounded open set $U$ with closure $\overline{U}$ contained in $D.$ If, in addition, $u(x)=0,$ for every $x\in D^c,$ we say that $u$ is singular $\alpha$-harmonic. It is called regular $\alpha$-harmonic in $D$ if (\ref{harmonic}) holds for $U=D.$ By the strong Markov property of $X_t,$ a regular $\alpha$-harmonic function is necessarily $\alpha$-harmonic. But the converse is not generally true. On the other hand, as in the classical case ($\alpha=2$) the $\alpha$-harmonicity can be descriped in terms of $\dal.$ Indeed, it is proved that a function $u\in\mathcal{L}_\alpha$ is $\alpha$-harmonic in $D$ if and only if it is continuous in $D$ and $\dal u=0$ in $D$ in the distributional sense (see for example \cite[Theorem 3.9]{bogdanby} for a detailed proof). We say that $u$ is $\alpha$-superharmonic in $D$ if 
\begin{equation}
E^x\bint\babsolu u(X_{\tau_U})\eabsolu\eint<\infty \quad\mbox{and}\quad u(x)\geq E^x[u(X_{\tau_U})]\quad;\quad x\in U,
\end{equation} 
for every bounded open set $U$ with closure $\overline{U}$ contained in $D.$

Let us denote $(X_t^D)$ the symmetric stable process killed upon exiting $D.$ It is well known that the transition density is given by
$$
p^D(t,x,y)=p(t,x,y)-r^D(t,x,y)\quad;\quad \ t>0,\ x,y\in D,
$$
where
$$r^D(t,x,y)=E^x\left[ p(t-\tau_D,X_{\tau_D},y),\tau_D< t\right].$$
The corresponding semigroup is then defined by
$$P^D_tf(x)=E^x\left[f(X_t) , t<\tau_D \right]=\int_D p^D(t,x,y)f(y)dy\quad;\quad x\in D,$$
for every Borel measurable function $f$ for which this integral
makes  sense.
A point $x\in\partial D$ is called regular for the set $D$ if $P^x(\tau_D=0)=1.$ The (open) bounded domain $D$ is called regular if all $x\in \partial D$ are regular for $D$ (for instance,  $C^{1,1}$-domains and domains satisfying the exterior cone condition are regular). In this case each function $f\in \cc{C}_b(D^c)$
 admits an  extension $H_D^\alpha f$ on $\rn$ such that $H_D^\alpha f$ is  regular $\alpha$-harmonic ~in ~$D$\cite{Landkof}. In other words,
  the function $h=H_D^\alpha f$ is the unique solution to the fractional Dirichlet problem
\[
\left\{
\begin{array}{rrrl}
\dal h&=&0& \mbox{in}\  D,\\
h&=&f&\textrm{in}\  D^c.
\end{array}
\right.
\]
Note that in the classical situation $(\alpha=2),$ by the continuity properties of Brownian motion, at the exit time from $D,$ one  necessarily is on $\partial D.$ But due to the jumping nature of the $\alpha$-stable process $(0<\alpha<2)$, at the exit time one could end up anywhere outside $D.$ That's why the natural Dirichlet boundary condition consists in assigning the value of $h$ in $D^c$ rather merely on $\partial D.$

For every $x\in D,$ the $\alpha$-harmonic measure relative to $x$ and $D,$
which will be denoted by $H_D^\alpha(x,\cdot),$
 is defined to be the positive Radon measure on $ D^c$ given by the mapping $f\mapsto H_D^\alpha f(x).$ It is proved in \cite{bogdan1997} that for $D$ say Lipschitz, $H_D^\alpha(x,\cdot)$ is concentrated on  $\overline{D}^c$ and is absolutely continuous with respect to the Lebesgue measure on $D^c.$ Furtheremore, the corresponding density function $P_D(x,y),$ $x\in D,$ $y\in D^c,$ is continuous in $(x,y)\in D\times$  ~$\overline{D}^c.$ In this situation, the solution of the Dirichlet problem can be expressed in term of the Poisson kernel $P_D$ as follows \cite{chensong1997}
 \[ H_D^\alpha f(x)=E^x\bint f(X_{\tau_D})\eint=\int_{D^c}P_D(x,y)\,f(y)\,dy\quad;\quad x\in D.\]
 
 The Green function $G_D^\alpha(\cdot,\cdot)$ of  a  domain $D\subset\rn$  is defined by 
 $$G_D^\alpha(x,y)=\int_0^\infty p^D(t,x,y)\,dt.$$
 Then $G_D^\alpha(x,y)$ is symmetric in $x$ and $y,$ $G_D^\alpha(x,y)$ is positive for $x,y\in D$ and continuous at $x,y\in\rn,$ $x\neq y.$ Also $G_D^\alpha(x,y)=0$ if $x$ or $y$ belongs to $D^c.$ Furtheremore, $G_D^\alpha(\cdot,y)$ is $\alpha$-harmonic in $D\sla \bacc y\eacc$ for every $y\in D$ and regular $\alpha$-harmonic in $D\sla B(y,r)$ for every $r>0.$  The Green function of the whole space $\rn,$ which is also called Riesz kernel, is given by 
 \[G^\alpha_{\rn}(x,y)=\frac{C_{N,\alpha}}{\babsolu x-y\eabsolu^{N-\alpha}}\] where $C_{N,\alpha}=2^{-\alpha}\pi^{-\frac{N}{2}}\Gamma(\frac{N-\alpha}{2})\babsolu\Gamma(\frac{\alpha}{2})\eabsolu^{-1}.$  Also, the explicit formula for the Green function of the ball $B_r=\bacc x\in\rn\ ;\ \babsolu x\eabsolu <r\eacc,\ r>0,$ is well known:
 \begin{equation}\label{green}
 G^\alpha_{B_r}(x,y)=\frac{C_{N,\alpha}}{\babsolu x-y\eabsolu^{N-\alpha}}\int_0^{\frac{\bp r^2-\babsolu x\eabsolu^2\ep\bp r^2-\babsolu y\eabsolu^2\ep}{\babsolu x-y\eabsolu^2}}\frac{s^{\frac{\alpha}{2}}}{(1+s)^{\frac{N}{2}}}\,ds\quad;\quad x,y\in B_r.
 \end{equation}
 
  Let $D$ be a bounded $C^{1,1}$ domain in $\rn.$ We denote by  $\delta(x):=\inf_{z\in\partial D}|x-z|$   the Euclidean distance
from $x\in D$ to the boundary of~$D$.   The following inequality was established in \cite{zha86} for $\alpha=2$ and in \cite{chensong1998} for $\alpha\in ]0,2[$.
\begin{equation}\label{green1}
   G_D^\alpha(x,y)\leq c \min \bacc\frac{1}{|x-y|^{N-\alpha}},\frac{\delta(x)^{\frac{\alpha}{2}}\delta(y)^{\frac{\alpha}{2}}}{|x-y|^N}\eacc,
  \end{equation}
  where $c>0$ is depending only on $N$ and $\alpha.$
  
 The Green operator $G_D^\alpha$ in an open set $D$ is defined, for every Borel measurable function~$f$ for which the following
   integral exists,  by
  \begin{equation}\label{eog}
  G_D^\alpha f(x)=\int_DG_D^\alpha(x,y)f(y)dy\quad;\quad x\in D.
  \end{equation}
   Hence
  $$
  G_D^\alpha f(x)=E^x\left[\int_{0}^{\tau_D}f(X_t)dt\right]=\int_{0}^{\infty}P_t^Df(x)dt\quad ;\quad x\in D.
  $$ 
We  recall that for every  $f\in \cc{B}_b(D)$,
 $G_D^\alpha f$ is a bounded continuous function on~$D$ satisfying $\lim_{x\rightarrow z}G_D^\alpha f(x)=~0$ for every  $z\in \partial D$ if we suppose further that $D$ is regular (all these properties follow by similar routine arguments to those in \cite{12} or \cite{5}).
 Moreover, it is simple to check that for every $f\in\mathcal{B}(\rn)$ such that $G_D^\alpha\babsolu f\eabsolu(x)<\infty$ for some $x\in\rn,$ we have
 \begin{equation}\label{operateurgreen}\dal G_D^\alpha f=-f\end{equation} in the distributional sense (see \cite[proposition 3.13]{bogdanby} or \cite[Lemma 5.3]{bogdandistribution}).
 
 Concluding this part of our preliminaries we refer the reader to \cite{hansen,bogdan2009,Landkof} for broader discussions on analytic counterparts of the above  definitions.
  \section{Nonnegative perturbation}
  
  In this section, we assume that  $\varphi:\bb{R}_+\rightarrow
\bb{R}_+$ is a continuous nondecreasing function such that
$\varphi(0)=0$ and $\rho:\rr^N\avaleur \rr_+$ is a locally bounded function. The aim is to characterize functions $\rho$ for which Eq. (\ref{eip}) has an entire bounded solutions. The outline is as follows. First, we prove that Eq. (\ref{eip}) has one and only one solution in a regular bounded domain $D$ coinciding with a given bounded continuous function on $D^c.$ After giving a sufficient condition to the existence of a nontrivial entire bounded solution of  Eq. (\ref{eip}), we investigate the special case when the function $\rho$ is radially symmetric and finally we return to the general case.
  \subsection{The Dirichlet problem in a bounded domain}
  
  Let $D$ be a bounded regular domain in $\rn.$ We consider the following fractional nonlinear problem
\begin{equation}\label{e31}
\left\{\begin{array}{rlll}\dal u&=&\rho(x)\,\varphi(u)& \mbox{in}\ D,\\
u&=&f&\mbox{in}\  D^c,\end{array}\right.
\end{equation}
where $f$ is a nonnegative bounded continuous function on~$ D^c.$ By a solution to the equation $\dal u=\rho\,\varphi(u)$ in a open set $U\subset\rn,$ we shall mean every real-valued continuous function $u$ on $U$ such that $\rho\varphi(u)$ is locally (Lebesgue) integrable on $U$ and the equality
\[\int_{\rn} u(x)\,\dal \theta(x)\,dx=\int_U \rho(x)\,\varphi(u(x))\,\theta(x)\,dx\]
holds for every nonnegative function $\theta\in\mathcal{C}_c^\infty(U)$.  Supersolutions and subsolutions to this  equation are to be understood in the same way replacing $"="$ by $"\leq"$ and $"\geq "$ respectively.

First, the following lemma states a straightforward  but an important result.
\begin{lem}\label{l1}  Let $u$ be a locally bounded nonnegative function in $\mathcal{B}(\rn).$
 The function $u$ is a solution of (\ref{eip}) in an open set $U\subset \rn$ if and only if $u+G_D^\alpha\bp \rho\varphi(u)\ep=H_D^\alpha u$ holds for every regular open set $D\subset \overline{D} \subset U.$
\end{lem}
\begin{proof} Taking into consideration the fact that $u$ is a solution of Eq. (\ref{eip}) in $U$ if and only if $u$ is a solution of Eq. (\ref{eip}) in each element of some covering of $U$ by open regular bounded subsets, we only need to prove that $u$ is a solution of Eq. (\ref{eip}) in $D$ if and only if $u+G_D^\alpha\bp \rho\varphi(u)\ep=H_D^\alpha u$ for  a regular bounded domain $D.$ To this end, fix a regular bounded domain $D$ and  define $h:=u+G_D^\alpha(\rho\,\varphi(u)).$ Since $u$ is bounded on $D,$ we deduce that $G_D^\alpha(\rho\,\varphi(u))\in\mathcal{C}_0(D).$ Whence, $h=u$ on $D^c$ and   $u$ is continuous on $D$ if only if $h$ is continuous on $D.$ On the other hand, using (\ref{operateurgreen}), for every $\theta\in \mathcal{C}_c^\infty(D),$  we have
\begin{displaymath}
\begin{split}
\int_{\rn} h(x)\dal \theta(x)\,dx&= \int_{\rn} u(x)\dal \theta(x)\,dx+\int_D G_D^\alpha(\rho\varphi(u))(x)\dal \theta(x)\,dx\\
&= \int_{\rn} u(x)\dal \theta(x)\,dx-\int_D \rho(x)\varphi(u(x)) \theta(x)\,dx.
\end{split}
\end{displaymath}
Therefore,   $u$ is a solution to Eq. (\ref{eip}) in $D$ if and only if  $\dal h=0$ in $D$ (or equivalently $h=H_D^\alpha u$) in the distributional sense. Hence, the lemma is proved.
\end{proof}

 The following  comparison principle  will be useful
to prove not only uniqueness but also the existence
 of a solution to  problem~(\ref{e31}).
 \begin{lem}\label{cprin} Let $\Psi\in\cc{B}(\bb{R})$ be a nondecreasing function and let
$u,v\in \cc{C}_b(\rn)$ such that
$$
\dal u\leq \rho(x)\Psi(u)\qquad \mbox{and} \qquad \dal v\geq \rho(x)\Psi(v)\qquad \mbox{in}\ D.
$$
If $u\geq v$ on $ D^c,$ then $u\geq v$ in $\rn.$
\end{lem}
 
 \begin{proof}  Define $w=u-v$ and suppose that the open set
$$
U=\left\{x\in D;w(x)<0\right\}
$$
 is not empty. Since $\Psi$ is nondecreasing, it is obvious that  $\dal w\leq \rho(x)\bp\Psi(u)-\Psi(v)\ep\leq 0$
 in~$U$, which means that  $w$ is $\alpha$-superharmonic in~$U.$ Furthermore, it is obvious that $w\geq 0$ on 
   $U^c\cap D$ and on $U^c\cap D^c$ we have also  $w\geq 0$ by hypothesis.
   The minimum principle for $\alpha$-superharmonic functions as stated in \cite{bet} (see also \cite{silvestre2007}) yields that    $w\geq 0$ in $U$     and this is absurd. Therefore~$U$ is
   empty. Hence $u\geq v$ in $D.$
\end{proof}
   
   The following lemma is already obtained in \cite{mahmoud} for the classical case $\alpha=2.$ We present here a readaptation to $\alpha\in ]0,2[.$ 
\begin{lem}\label{l221}
For every $M>0,$ the family  $\left\{G_D^\alpha u;\left\|u\right\|\leq
M\right\}$ is relatively compact in $\cc{B}_b(D)$ with respect to the uniform
 norm.
\end{lem}
\begin{proof}
First, we note  that $x\mapsto G_D^\alpha1(x)=E^x[\tau_D]$ is dominated
on~$D$ by $c\,m(D)^{\frac{\alpha}{N}}$ for some constant $c>0$ ( this follows from a direct modification of the proof of \cite[theorem 1.17]{5} ). Here and in all the following,  $m$ denotes the Lebesgue measure in $\rn.$  Consequently for every $u$ such that $\left\|u\right\|\leq~
M$ we get
$$
\left\|G_D^\alpha u\right\|\leq M \sup_{x\in D}E^x[\tau_D].
$$
Thus the family $\left\{G_D^\alpha u;\left\|u\right\|\leq M\right\}$ is
uniformly bounded.  
 Next, we claim  that the family
 $\left\{G_D^\alpha(x,\cdot); x\in D\right\}$ is uniformly integrable.
 Indeed, by (\ref{green1}) there exists $c_1>0$ such that for every Borel subset $A$ of $D$ and every $\eta_0>0$ small enough we have
\begin{displaymath}
\begin{split}
\int_A G_D^\alpha(x,y)dy&\leq c_1 \int_A \frac{dy}{\left|x-y\right|^{N-\alpha}}\\
&\leq
c_1\int_{B(x,\eta_0)}\frac{dy}{\left|x-y\right|^{N-\alpha}}+c_1\int_{A\sla
B(x,\eta_0)}
\frac{dy}{\eta_0^{N-\alpha}}\\
&\leq c_2\bp\eta_0^\alpha+\frac{m(A)}{\eta_0^{N-\alpha}}\ep,
\end{split}
\end{displaymath} where $c_2>0$ depends only on $N$ and $\alpha.$ Let $\varepsilon>0$ and choose $\eta_0$ so that $c_2\eta_0^\alpha<\varepsilon.$ Put $\eta=(\frac{\varepsilon}{c_2}-\eta_0^\alpha)\eta_0^{N-\alpha}.$
 Then  for every  Borel subset $A$ of $D$  such that
$m(A)<\eta$ we have
$$
\int_A G_D^\alpha(x,y)\, dy\leq \varepsilon.
$$
Hence, the uniform integrability of the family $\left\{G_D^\alpha(x,\cdot);
x\in D\right\}$ is shown. Therefore, in virtue of Vitali's
convergence theorem (see, e.g, \cite{rud}),  we conclude that for
every $z\in D$,
\begin{displaymath}
\begin{split}
\lim_{x\to z}\sup_{\|u\|\leq M}\left|\int_D G_D^\alpha(x,y)u(y)dy-\int_DG_D^\alpha(z,y)u(y)dy\right|\\
\qquad\qquad\leq M \lim_{x\to
z}\int_D\left|G_D^\alpha(x,y)-G_D^\alpha(z,y)\right|dy= 0.
\end{split}
\end{displaymath}
This means that   the family $\left\{G_D^\alpha(x,\cdot); x\in D\right\}$
is equicontinuous  which finishes the proof of the lemma.
\end{proof}
For $\alpha=2,$ existence of solutions to semilinear Dirichlet problems of
kind~(\ref{e31})  was widely studied in the literature considering
 various hypotheses on the function~$\varphi$ (see, e.g.,
 \cite{baha02,mahmoud,dyn00,khalifabounded,8}). Under the hypothesis mentioned in the begining of the current section we get, for $\alpha\in]0,2[,$ the following result.

\begin{pro}\label{thm81}
For every $f\in \cc{C}_b^+( D^c),$ there exits one and only one
function $u\in\cc{C}^+_b(\rn)$ solution to
problem~(\ref{e31}). Moreover, for every $x\in\rn,$
\begin{equation}\label{relation} u(x)+G_D^\alpha(\rho\,\varphi(u))(x)=H_D^\alpha f(x)\quad;\quad x\in\rn.\end{equation}
\end{pro}
\begin{proof}  We  observe that, by
the comparison principle (Lemma~\ref{cprin}),  problem~(\ref{e31})
possesses at most one solution.  To prove the
existence, take $f\in \cc{C}_b^+(
D^c),$ $a=\left\|f\right\|$, $M=a+\varphi(a)\sup_{x\in D} \rho(x)\sup_{x\in D}E^x[\tau_D]$
and define  $
\Lambda=\left\{u\in
\cc{C}(\overline{D});\left\|u\right\|\leq M\right\}.$
 Let $h=H_D^\alpha f$ and consider the operator
$T:\Lambda\rightarrow
\cc{C}(\overline{D})$  defined by
$$Tu(x)=h(x)-E^x\left[\int_{0}^{\tau_D}\rho(X_s)g(u(X_s))ds\right]\quad ;\quad \ x\in D,$$
where~$g$ is the real-valued odd function given by
$g(t)=\inf(\varphi(t),\varphi(a))$ for every   $t\geq 0$. Since
$\left|g(t)\right|\leq
\varphi(a)$ for every $t\in\bb{R},$ we get
$$\left|Tu(x)\right|\leq M$$
for every $x\in D$ and every $u\in\Lambda.$ This implies that
$T(\Lambda)\subset \Lambda.$ Now, let  $(u_n)_{n\geq 0}$ be a
sequence in $\Lambda$ converging uniformly to
  $u\in\Lambda$ and  let $\varepsilon >0.$ Since $g$ is uniformly continuous in
  $[-M,M],$ we deduce that there exists
    $n_0\in\mathbb{N}$ such that for every $n\geq n_0$ 
$$\left|g(u_n(X_s))-g(u(X_s))\right|<\varepsilon,\quad\text{for all } s\in[0,\tau_D].$$
It follows that for every  $n\geq n_0$ and  $x\in D$,
\begin{displaymath}
\begin{split}
\left|Tu_n(x)-Tu(x)\right|&=\left|E^x\left[\int_{0}^{\tau_D}\rho(X_s)g(u_n(X_s))ds\right]
-E^x\left[\int_{0}^{\tau_D}\rho(X_s)g(u(X_s))ds\right]\right|\\
&\leq E^x\left[\int_{0}^{\tau_D}\rho(X_s)\left|g(u_n(X_s))-g(u(X_s))\right|ds\right]\\
&\leq \varepsilon \sup_{x\in D}\rho(x) \sup_{x\in D}E^x[\tau_D].
\end{split}
\end{displaymath}
This shows that  $(Tu_n)_{n\geq 0}$ converges uniformly to $Tu.$ We
then conclude that $T$ is a continuous operator. On the other hand,
$\Lambda$ is a closed bounded convex subset of
$\cc{C}(\overline{D}).$ Moreover, in virtue of Lemma~\ref{l221},
$T(\Lambda)$ is relatively compact. Thus, the Schauder's fixed point
theorem ensures  the existence of  a function  $u\in \Lambda$ such
that $u=h-G_D^\alpha (\rho\,g(u)).$ Applying the comparison principle, it follows
that  $0\leq u\leq a$  and so  $g(u)= \varphi(u).$ We then get  immediately (\ref{relation}). Hence, the proof
is finished by Lemma \ref{l1}.
\end{proof}

\subsection{ A sufficient condition}

The unique solution to  problem (\ref{e31}) will be always   denoted
by $H_D^{\alpha,\varphi} f.$

Our purpose now constists in studying the existence of  nontrivial bounded  solutions of  the Eq. (\ref{eip})
in the whole space $\rn $ (entire solutions). 

Applying the comparison principle as stated in Lemma \ref{cprin}, we obtain the following elementary results.
\begin{lem}\label{pcg}    Let $D$ and  $D'$  be  regular  open sets such that $D'\subset\overline{D'}\subset D\subset \overline{D}\subset \rn.$ 
\benu
\item[(a)] If $f,g\in \cc{C}_b^+(D^c)$ such that $f\leq g$ then $H_D^{\alpha,\varphi}f\leq \hd g.$ 
 \item[(b)] If $u\in \cc{C}^+(\rn)$ is a supersolution of (\ref{eip}) in $\rn,$ then $\hd u\leq H_{D^{'}}^{\alpha,\varphi}u\leq u.$
 \eenu
\end{lem}

It is noteworthy that the monotony and not the sign of $\varphi$ is important to establish either  Lemma \ref{l1} or Lemma \ref{pcg}. So the results remain true even in the negative perturbation case that we shall consider in the next section.

In the following proposition, we give a sufficient condition for the existence of a nontrivial entire bounded solution to Eq. (\ref{eip}).
\begin{pro}\label{sb} Assume that for some $x_0\in\rn,$
\begin{equation}\label{c11}
\int_{\rn}\frac{\rho(y)}{\babsolu x_0-y\eabsolu^{N-\alpha} }\,dy<\infty.  
\end{equation}
 Then Eq. (\ref{eip}) admits a nonnegative nontrivial entire bounded solution. 
\end{pro}
\begin{proof}  Let $\lambda>0$  and define  $u_k=H_{B_k}^{\alpha,\varphi}\lambda$ for every integer $k\geq 1$ where $B_k=B(0,k) $ is the ball of center $0$ and radius $k.$ Then, by statement  (b) in Lemma \ref{pcg}, $(u_k)$ is a nonincreasing sequence of continuous functions since $\lambda$ is a supersolution of Eq. (\ref{eip}) in $\rn$. Further, by (\ref{relation}), 
\begin{equation}\label{aux3}
u_k+G_{B_k}^{\alpha}(\rho\varphi(u_k))=\lambda
\end{equation}
  for every $k\geq1.$ Besides, noting  that $u_k\leq\lambda $ for every $k\geq1,$ we deduce that  the limit function $u:=\lim_{k\avaleur\infty}u_k$ exists and is bounded above by $\lambda.$ Next, we tend to prove that $u$ is nontrivial.  Seing that for every $k\geq 1$
\[\babsolu G_{B_k}^\alpha(x_0,y)\rho(y)\varphi(u_k(y))\eabsolu\leq \varphi(\lambda) G_{\rn}^\alpha(x_0,y)\rho(y),\]
 by (\ref{c11}) and the dominated convergence theorem, letting $k$ tend to $\infty$ in (\ref{aux3}), we obtain that $u(x_0)+G^\alpha_{\rn}(\rho\,\varphi(u))(x_0)=\lambda.$  Hence $u\not\equiv 0$ since $\varphi(0)=0.$ 
   It remains to check that $u$ is a solution of Eq. (\ref{eip}). Let $D$ be  an arbitrary regular bounded domain in $\rn.$ It exists $k_0\geq1$ such that $\overline{D}\subset B_k$ for every $k\geq k_0.$ By Lemma \ref{l1}, $u_k$ satisfies the following integral equality
\begin{equation}\label{a}
u_k+G_D^\alpha(\rho\,\varphi(u_k))=H_D^\alpha u_k
\end{equation}
for every $k\geq k_0.$ Letting $k$ tend to $\infty$ in (\ref{a}) we obtain by the dominated convergence theorem that 
\[u+G_D^\alpha(\rho\,\varphi(u))=H_D^\alpha u.\]
Again, in virtue of Lemma  \ref{l1}, the arbitrariness of the domain $D$ implies that $u$ is a solution to Eq. (\ref{eip}) in the whole space $\rn$ as desired.
 
 \end{proof}
 In several places in this paper we will use the following remark.
 \begin{rmq}\label{solution}
 We would like to mention that we can learn from the above proof that, for every $\lambda>0$, $\inf_{k\geq 1}H_{B_k}^{\alpha,\varphi} \lambda$ is a nonnegative entire bounded solution of Eq. (\ref{eip})  but we do not  guarantee that it is nontrivial.
 \end{rmq}
 \subsection{The radial case}
 In this section, we shall discuss the radial case. But before, we need the following lemma which is available for $\rho$ which is not  necessarily radially symmetric. 
\begin{lem}\label{equationintegrale} Suppose that Eq. (\ref{eip}) admits a nonnegative entire bounded solution $u.$ 
  Then  for every $x\in \rn,$
 \[u(x)+G_{\rn}^\alpha(\rho\,\varphi(u))(x)=\bnorme u\enorme.\]
 \end{lem}
 \begin{proof} By Lemma \ref{l1}, for every $k\geq 1$ we have
 \begin{equation}\label{0}
 u(x)+G_{B_k}^\alpha(\rho\,\varphi(u))(x)=H_{B_k}^\alpha u(x)\quad;\quad x\in\rn.
 \end{equation}
 One easily observe that $(H_{B_k}^\alpha u)_{k\geq1}$ is  uniformily bounded  above by $\bnorme u\enorme$. Since, $u$ is $\alpha$-subharmonic in $\rn,$ it follows that $H_{B_k}^\alpha u\leq H_{B_{k+1}}^\alpha u$ for every $k\geq1.$ Consequently, the limit function $h:=\sup_{k\geq 1} H^{\alpha}_{B_k}u$ exists and it is a $\alpha$-harmonic function in the whole space $\rn,$ which  in turn means, by the Liouiville property \cite{bogdanliouville,chenliouville}, that $h=c$ for some  nonnegative constant $c\in\rr$. Letting $k$ tend to $\infty$ in (\ref{0}) and recalling that $\sup_{k\geq 1}G_{B_k}^\alpha=G_{\rn}^\alpha$, we obtain   
 \[v:=G_{\rn}^\alpha(\rho\,\varphi(u))=c-u\qquad\text{in }\rn.\]
 Since $v$ is a potential in $\rn$, we deduce that $\displaystyle{\inf_{x\in\rn} v(x)=0}$ and so $\displaystyle{c-\sup_{x\in\rn}u(x)=0}$ which completes the proof.
 \end{proof}
 We notice, under the hypothesis $\lim_{\babsolu x\eabsolu\to\infty}G_{\rn}^\alpha \rho(x)=0,$ that every nonnegative bounded solution to Eq. (\ref{eip}) admits a limit in $\infty,$ namely, $\lim_{\babsolu x\eabsolu\to\infty}u(x)=\bnorme u\enorme.$ 
 
 The proof of the following theorem uses elements of the corresponding proof from \cite{khalifabounded}. Nevertheless, the maximum principle exploited there for radially symmetric subharmonic functions (relative to the classical Laplacian) does not apply for those relative to fractional Laplacian.
 \begin{thm}\label{radial} Assume that $\rho$ is radially symmetric on $\rn.$ Then the following statements are equivalent. 
 \benu
 \item[(a)] $\displaystyle{\int_0^\infty r^{\alpha-1}\rho(r)\,dr<\infty.}$
 \item[(b)] It exists $x_0\in\rn$ such that $G_{\rn}^\alpha\rho(x_0)<\infty.$
 \item[(c)] Eq.(\ref{eip}) admits a nonnegative nontrivial entire bounded solution.
 \item[(d)] Eq.(\ref{eip}) admits a nonnegative nontrivial entire bounded solution which is radially symmetric.
 \eenu
 \end{thm}
 \begin{proof} Using the spherical coordinates, we easily get that
 \[G_{\rn}^\alpha(0)=C_{N,\alpha}\int_0^\infty r^{\alpha-1}\rho(r)\,dr\]
 and then (a) implies (b)  holds, while (b) implies (c) is already obtained in proposition \ref{sb}. To prove (c) implies (d), let $u$ be a nontrivial entire bounded solution of Eq. (\ref{eip}) and choose $\lambda\geq \bnorme u\enorme.$ Define for every $k\geq 1,$ $v_k=H_{B_k}^{\alpha,\varphi}\lambda.$  It is not hard to see that $v_k$ is radially symmetric but we will spell out the details. By proposition \ref{thm81} 
 \begin{equation}\label{int}
 v_k(x)=\lambda-G_{B_k}^\alpha\bp\rho\,\varphi(v_k)\ep (x)\quad ;\quad x\in\rn.
 \end{equation} 
 Let $\kappa$ be an orthogonal transformation in $\rn.$ Recalling the explicit formula (\ref{green}) of the Green function in the ball, we get that $G_{B_k}^\alpha\bp\rho\,\varphi(v_k)\ep\bp \kappa(x)\ep=G_{B_k}^\alpha\bp\rho\,\varphi(v_k\circ\kappa)\ep\bp x\ep$ for every $x\in\rn$ and consequently, by (\ref{int}), we have
 \[v_k\circ\kappa(x)=\lambda-G_{B_k}^\alpha\bp\rho\,\varphi(v_k\circ\kappa)\ep\bp x\ep\quad;\quad x\in\rn.\] 
 The uniqueness of nonnegative solution of the problem (\ref{e31}) yields that $v_k=v_k\circ\kappa$ and hence $v_k$ is  radially symmetric. On the other hand, by Lemma \ref{pcg} we get that $u\leq v_{n+1}\leq v_n\leq \lambda$ for every $k\geq 1$ . Put $v:=\inf_{k\geq 1}v_k.$ Then $v$ is radially symmetric and bounded below by $u$ which is non identically zero. Also, by Remark \ref{solution} $v$ is an entire solution of Eq. (\ref{eip}). Finally, to prove that (d) implies (a), let $v$ be a nontrivial entire bounded radially symmetric solution of Eq. (\ref{eip}). Then in virtue of Lemma 
 \ref{equationintegrale}, we have    
 \[v(x)+G^\alpha_{\rn}(\rho\varphi(v))(x)=\bnorme v\enorme\quad;\quad x\in \rn,\]
 and in particular,
 \begin{equation}\label{e1}
 v(0)+c\int_{0}^{\infty}r^{\alpha-1}\rho(r)\varphi(v(r))\,dr=\bnorme v\enorme,
 \end{equation}
 for some constant $c>0.$ Moreover, by \cite[Theorem 2]{mizuta}, $\lim_{\babsolu x\eabsolu\to\infty} G^\alpha_{\rn}(\rho\varphi(v))(x)=0$ since $G^\alpha_{\rn}(\rho\varphi(v))$ is radially symmetric. Hence, $\lim_{\babsolu x\eabsolu\to\infty}v(x)=\bnorme v\enorme.$ It follows that it exists $r_0>0$ such that $v(r)\geq \frac{1}{2}\bnorme v\enorme$ for every $r\geq r_0.$ Then, according to (\ref{e1}), we get
 \begin{displaymath}
 \begin{split}
 C_{N,\alpha}\varphi(\frac{1}{2}\bnorme v\enorme)\int_{r_0}^\infty r^{\alpha-1}\rho(r)\,dr&\leq c\int_{r_0}^\infty r^{\alpha-1}\rho(r)\varphi(v(r))\,dr\\
 &\leq c\bp\bnorme v\enorme-v(0)\ep<\infty.\end{split}
 \end{displaymath} The fact that $\rho$ is locally bounded yields that $\int_0^{r_0}r^{\alpha-1}\rho(r)\,dr<\infty$ and hence, (d) implies (a) holds.
 \end{proof}
 \subsection{Transient sets}
 We have seen above that (\ref{c11}) is a sufficient condition for the existence of a bounded solution to (\ref{eip}) and that in the radial case it is a necessary condition as well. A natural question which needs to be raised here if (\ref{c11}) is necessary when $\rho$ is not radially symmetric. To settle this question, we need some preparation. So one can see this section as a " addendum" but  we shall prove on the way a much more general result which is of interest in itself.
 
   We shall first  clarify some terminology. Let $A$ be a Borel set and let $T_A$ be the first hitting time of  $A$
\[T_A:=\bacc t>0\ ;\ X_{t}\in A\eacc.\]
The  set $A$ is said to be $\alpha$-recurrent if $P^x(T_A<\infty)=1$ for every $x\in \rn$ and $\alpha$-transient otherwise, that is if  there exists $x_0\in\rn$ such that $P^{x_0}(T_A<\infty)\neq 1$ (\cite[p. 24]{10} or \cite[p. 121]{chung}).
Let $u$ be a positive $\alpha$-superharmonic function in  $\rn.$  We denote  the set of all  nonnegative $\alpha$-superharmonic functions in $\rn$ by $\mathcal{S}^+. $ The regularized reduced function (or balayage) of $u$ relative to $A$ in $\rn$ is given by
\[\hat{R}_u^A(x)=\liminf_{y\avaleur x}R_u^A(y)\quad;\quad x\in \rn,\]
where \begin{displaymath}
\begin{split}
R_u^A(x)&= \inf \bacc v(x)\ ;\ v\in \mathcal{S}^+\text{ and } v\geq u \text{ on } A\eacc\\
&= \inf \bacc v(x)\ ;\ v\in \mathcal{S}^+,\ v=u\ \text{on }A\text{ and } v\leq u \text{ on } \rn\eacc.
\end{split}
\end{displaymath}
It is well known (see \cite[p. 263]{hansen} or \cite[p. 231]{generaltheory}) that  
\[\hat{R}_u^A(x)=P_{T_A}u(x):=E^x\bint u(X_{T_A})\ ;\ T_A<\infty\eint\quad;\quad x\in \rn.\]
Hence the following assertions are obviously equivalent.
\begin{description}
\item[(a)] $A$ is transient.
\item[(b)] For some $\lambda>0,$ $\hat{R}_\lambda^A\neq \lambda.$ 
\item[(c)] For every $\lambda>0,$ $\hat{R}_\lambda^A\neq \lambda.$ 
\item[(d)] For every $\lambda>0,$ there exists $s\in\cc{S}^+$ such that $s\geq\lambda$ on $A$ and $s\not\geq \lambda$ on $\rn$  (that is $A$ is thin at $\infty$ in the sense of \cite[p. 215]{armitage}).
\item[(e)] For every $\lambda>0,$ there exists $s\in\cc{S}^+$ such that $s=\lambda$ on $A$ and $s(x_0)< \lambda$ for some $x_0\in\rn.$
\end{description}
Although, as evoked in the introduction, many properties of the classical case $'\alpha=2'$ can be more or less readily extended for $0<\alpha<2,$ it seems that there are no one clear reason for a $2$-transient set to be   $\alpha$-transient. One may prove this result stochastically by analysing closely the properties of the $\alpha$-stable process. Nevertheless, the approach that we have adopted consists of verifying that every $2$-superharmonic function is $\alpha$-superharmonic as weel and this fact  is clearly much more than that we need to prove and gives answer to our question. We believe that this approach is more relevant for the framework of our paper.
\begin{thm}
 Let $u$ be a nonnegative function  in $L^\infty_{loc}(\rn)\cap\mathcal{L}_\alpha. $ 
If $u$ is $2$-superharmonic on $\rn$ then $u$ is $\alpha$-superharmonic on $\rn.$
\end{thm}
\begin{proof}
We split the proof into two steps. In the first place we suppose additionally that $u\in C^2(\rn)$.  Then for every $x\in \rn$ we have
\[
\dal u(x)=c_{N,\alpha}\int_{\rn}\frac{u(x+y)-u(x)}{|y|^{N+\alpha}}\,dy.
\]
Let $S^{N-1}$ denotes the unit sphere of $\rn$ and let $\sigma$ denotes the surface area measure on $S^{N-1}$.
Using spherical coordinates in $\rn$, we get
\[
\dal u(x)=c_{N,\alpha}\int_0^\infty\frac{w(x,r)}{r^{\alpha+1}}\,dr,
\]
where
\[
w(x,r)= \int_{S^{N-1}}u(x+ry)\sigma(dy)-\sigma\left(S^{N-1}\right)\,u(x).
\]
Since $u$ is $2$-superharmonic on $\rn$, we have $w(x,r)\leq0$ for all $r\geq 0$ and for all $x\in \rn$. Therefore, $\dal u\leq0$ on $\rn$ which means that $u$ is $\alpha$-superharmonic on $\rn$. Now, we turn to the general case where $u$ is lower semi-continuous on $\rn$ and not necessarily of class $C^2$.
 Obviously, in order to prove that $u$ is $\alpha$-superharmonic on $\rn$, it is sufficient to show that $H_D^\alpha u\leq u$  for every  regular bounded open set $D$. Consider  the approximate identity $\phi$ defined on $\rn$ by
 \[
 \phi(x)= c\,e^{\frac{1}{|x|^2-1}}\; \textrm{ if }\; |x|<1\;\;\textrm{ and }\;\; \phi(x)=0\;\textrm{ if }\; |x|\geq 1,
 \]
 where the constant $c>0$ is chosen so that $\int_{\rn}\phi(x)\,dx=1$. For every $n\geq 1$, let $\phi_n$ be the function defined on $\rn$ by $\phi_n(x)=n^N\phi(nx)$. Obviously, for every $n\geq 1$, $\phi_n\in \mathcal{C}_c^\infty(\rn)$ and with support in the closed ball $\overline{B}(0, 1/n)$. Next, for every $n\geq 1$, we define
 $$
 u_n(x)=\int_{\rn}u(y)\phi_n(x-y)\,dy\quad;\quad x\in \rn.
 $$
 Using the spherical coordinates and the fact that $u$ is $2$-superharmonic  we get, for every $x\in \rn$,
\begin{eqnarray*}
 u_n(x)&=&\int_{\rn}u(x-y)\phi_n(y)\,dy\\
 &=&\int_0^{\frac 1n}t^{N-1}\left(\int_{S^{N-1}}u(x-tz)\sigma(dz)\right)\phi_n(t)\,dt\\
&\leq & u(x)\,\sigma\left(S^{N-1}\right)\,\int_0^{\frac 1n}t^{N-1}\phi_n(t)\,dt \\
&=& u(x).
\end{eqnarray*}
This shows in particular that $u_n\in\mathcal{L}_\alpha$ for every $n\geq1.$
Also, it follows that $\limsup_nu_n(x)\leq u(x)$ for all $x\in \rn$.
 On the other hand, 
  using Fatou's Lemma and the fact that $u$ is lower semi-continuous on $\rn$, we obtain
 $$
 \liminf_nu_n(x)=\liminf_n\int_{\rn}u(x-\frac yn)\phi(y)\,dy\geq u(x).
 $$ 
  Hence, for every $x\in \rn$, $\liminf_nu_n(x)=\limsup_nu_n(x)=u(x)$ which means that the sequence $(u_n(x))_n$ converges to $u(x)$.
 Since $u\in L^\infty_{loc}(\rn)$ and $\phi_n$ and all its partial derivatives are bounded on $\rn$ and vanish outside $B(0,1/n)$, we see that $u_n\in C^\infty(\rn)$. Also, for every $x\in \rn$,
 $$
 \Delta u_n(x)=\int_{\rn} u(y) \Delta(\phi_n(\cdot-y))(x)\,dy=\int_{\rn}u(y)\Delta(\phi_n(\cdot-x))(y)\,dy.
 $$
 This implies that $\dal u_n(x)\leq0$ since,
by  hypothesis, $u$ is $2$-superharmonic on $\rn$ which is equivalent to $\Delta u\leq 0$ in the distributional sense. Thus it follows from the first step that $\dal u_n(x)\leq 0$ for every $x\in \rn$, or equivalently, for every bounded regular open set $D$
$$
H_D^\alpha u_n\leq u_n \;\textrm{ on }\;\rn.
$$
Hence, letting $n$ tend to $\infty$, we deduce  that $H_D^\alpha u\leq  u$ since $u_n\to u$  on $\rn$ as $n\to \infty$.
\end{proof}
The following important result is an immediate consequence of the above theorem.
\begin{cor}
Every  $2$-transient set $A\subset\rn$   is $\alpha$-transient.
\end{cor}
\begin{exemple}\label{exemple1}
Assume that $N>3$. Let $\beta>1/(N-3)$ and let $h(r)=r/(\ln r)^\beta$ for $r\geq e$. Consider the thorn $A$ given by
$$
A=\{(x_1,...,x_N)\in\rn\ ;\  x_1\geq e\;\textrm{ and }\; x_2^2+\cdots+ x_N^2\leq h^2(x_1)\}.
$$
It was shown in \cite[Proposition 3.3.6]{10} that $A$ is $2$-transient. Therefore, by the above theorem, the thorn $A$ is also $\alpha$-transient.
\end{exemple}
 \subsection{The general case}
 In this section we shall discuss the general case when $\rho$ is not necessarily radially symmetric. The proof of the main result requires some additional preparation.
 \begin{lem}\label{sommerho} Let $\rho_1,\rho_2$ be nonnegative locally bounded functions in $\rn.$ Suppose that Eq. (\ref{eip}) admits a nonnegative nontrivial entire bounded solution for $\rho=\rho_1$ and for $\rho=\rho_2.$ Then,  Eq.  (\ref{eip}) admits a such solution for $\rho=\rho_1+\rho_2.$
\end{lem}
\begin{proof} Let $u_{\rho_1}$ (resp. $u_{\rho_2}$ ) be a nonnegative nontrivial entire bounded solutions to Eq. (\ref{eip}) for $\rho=\rho_1$ (resp. $\rho=\rho_2$ ). We recall from Lemma \ref{equationintegrale} that 
\begin{equation}\label{rho1}
u_{\rho_i}+G^\alpha_{\rn}(\rho_i\,\varphi(u_{\rho_i}))=\bnorme u_{\rho_i}\enorme\qquad i=1,2.
\end{equation}
As before, we denote by $B_k$ the ball of center $0$ and the radius $k$ $(k\geq 1)$.
 Put $\lambda:=\max_{i=1,2}\bnorme u_{\rho_i}\enorme>0.$ For every $k\geq1,$ define $u_k$ and $v_k$   as follows
\[\bacc\ba{rcll}\dal u_k&=&\rho_1\,\varphi(u_k)&\text{ in } B_k\\
u_k&=&\lambda&\text{ in }B_k^c\ea\right.\qquad;\qquad\bacc\ba{rcll}\dal v_k&=&\rho_2\,\varphi(v_k)&\text{ in } B_k\\
v_k&=&\lambda&\text{ in }B_k^c\ea\right.\]
Then by Lemma \ref{pcg}, for every $k\geq 1$ we have $u_{\rho_1}\leq u_k\leq \lambda$ and $u_{\rho_2}\leq v_k\leq \lambda.$ 
 For every $k\geq 1$ let $w_k$ be such that
 \[\bacc\ba{rcll}\dal w_k&=&\rho\,\varphi(w_k)&\text{ in } B_k\\
w_k&=&\lambda&\text{ in }B_k^c\ea\right.\quad\]
   where $\rho=\rho_1+\rho_2.$ By Remark \ref{solution}, $w:=\inf_{k\geq1}w_k$ is an entire bounded solution of Eq. (\ref{eip}). We claim   that $w\neq0$ and this achieves the proof. Indeed, in virtue of the comparison principle $w_k\leq \inf (u_k,v_k)$ and therefore
\[\dal (\lambda+w_k-u_k-v_k)=(\rho_1+\rho_2)\varphi(w_k)-\rho_1\varphi(u_k)-\rho_2\varphi(v_k)\leq 0\text{ in }B_k.\]
Hence, $\lambda+w_k-u_k-v_k$ is a $\alpha$-superharmonic function in $B_k.$ Note also that $\lambda+w_k-u_k-v_k=0$ on $B_k^c.$ We deduce from the minimum principle (for $\alpha$-superharmonic functions) that $\lambda+w_k-u_k-v_k\geq0$ in $\rn.$ Next, without lose of generality, one may suppose that $\lambda=\bnorme u_{\rho_1}\enorme.$ Seeing that
\[\lambda-u_{\rho_1}+w_k\geq \lambda-u_k+w_k\geq v_k\geq u_{\rho_2},\]
we deduce by (\ref{rho1}) (for $i=1$) that 
\begin{equation}\label{rho2}
G_{\rn}^\alpha(\rho_1\varphi(u_{\rho_1}))+w_k\geq u_{\rho_2}.
\end{equation}
Suppose, contrary to our claim, that $w=0.$ By letting $k$ tend to $\infty$ in (\ref{rho2}) and using again (\ref{rho1}) (for $i=2$) we obtain
\[\xi:=G_{\rn}^\alpha(\rho_1\varphi(u_{\rho_1}))+G_{\rn}^\alpha(\rho_2\varphi(u_{\rho_2}))\geq \bnorme u_{\rho_2}\enorme>0,\]
which leads to an absurdity because $\xi$ is a potential. Hence, the claim is checked.  
\end{proof}

Now we are in position to characterize all nonnegative functions $\rho\in L^\infty_{\text{loc}}(\rn)$ for which Eq. (\ref{eip}) admits a nontrivial bounded solution.
\begin{thm}\label{cns} Eq. (\ref{eip}) admits a  nonnegative nontrivial entire  bounded  solution in $\rn$ if and only if there exists a transient  set $A\subset \rn$ and $x_0\in\rn$ such that 
$$\int_{ A^c}\frac{\rho(y)}{\babsolu x_0-y\eabsolu^{N-\alpha} }\,dy<\infty.$$
\end{thm}
\begin{proof} To prove the  sufficiently, we write $\rho$ as a sum $\rho=\rho_1+\rho_2$ where $\rho_1=1_A\rho$ and $\rho_2=1_{ A^c}\rho.$ In virtue of proposition \ref{sb}, Eq (\ref{eip}) admits a nonnegative nontrivial entire bounded solution $u_{\rho_2}$ for $\rho=\rho_2.$

 Now, let $\lambda>0$ and let $s_0$ be an $\alpha$-superharmonic function  such that $s_0\geq\lambda$ on $B$ but not bounded below by $\lambda$ on $\rn.$   For every $k\geq 1$ define the function $u_k$ as follows 
\[\bacc\ba{rcll}\dal u_k&=&\rho_1\varphi(u_k)&\text{ in }B_k\\
u_k&=&\lambda&\text{ in }B_k^c\ea\right.\]
  The equality $u_k+G_{B_k}^\alpha\bp\rho_1\,\varphi(u_k)\ep=\lambda$ (on $B_k$) implies that $G_{B_k}^\alpha(\rho_1\,\varphi(u_k))\leq \lambda=R_\lambda^A$ on $A\cap B_k.$ Since $\bacc \rho_1>0\eacc\subset A,$ by the domination principle (\cite[p.203]{hansen},\cite[p. 166]{helms},\cite[p. 175]{10}), we obtain that $G_{B_k}^\alpha(\rho_1\,\varphi(u_k))\leq R_\lambda^{A}$ on $ B_k.$ This proves that $u_k+R_\lambda^{A}\geq\lambda$ on $B_k,$ while on $B_k^c,$ $u_k=\lambda$ and consequently $u_k+R_\lambda^{A}\geq\lambda$ on $\rn.$ Then we get  that $u_{\rho_1}:=\inf_{k\geq1}u_k\geq \lambda-R_\lambda^{A}\not\equiv 0.$ Therefore,  by Remark \ref{solution}, $u_{\rho_1}$ is a nonnegative nontrivial entire bounded  solution of Eq. (\ref{eip}) for $\rho=\rho_1.$ Whence, the  Lemma \ref{sommerho} finishes the if part.

Let us now prove the necessity. Let   $u$ be a nontrivial entire bounded solution of Eq. (\ref{eip}). By Lemma \ref{equationintegrale}, $u$ satisfies
\[ \bnorme u\enorme=u+\int_{\rn}G_{\rn}^\alpha(\cdot,y)\rho(y)\varphi(u(y))\, dy.\]
 Let $A=\bacc 2u\leq \bnorme u\enorme\eacc.$  Define $s:=2(\bnorme u\enorme-u).$ It is clear that $s$  is a nonnegative $\alpha$-superharmonic function in $\rn,$  $s\geq\bnorme u\enorme$ in $A$ but  $s\not\geq \bnorme u\enorme$ in $\rn.$ So the set $A$ is transient. Furthermore,
\begin{displaymath}
\begin{split}\varphi(\frac{\bnorme u\enorme}{2})\int_{A^c}G_{\rn}^\alpha(\cdot,y)\, \rho(y)\,dy&\leq \int_{A^c}G_{\rn}^\alpha(\cdot,y)\, \rho(y)\,\varphi(u(y))\,dy\\
&\leq \int_{\rn}G_{\rn}^\alpha(\cdot,y)\, \rho(y)\,\varphi(u(y))\,dy\\
&\leq \bnorme u\enorme,
\end{split}
\end{displaymath}
which finishes the proof.
\end{proof}
We   summarize some of the obtained  results as follows.
\begin{cor}\label{corollaire} Under each of the following conditions, Eq. (\ref{eip}) has a nonnegative nontrivial entire bounded solution.
\benu
\item[(a1)] It exists a transient set  $A$ such that $\bacc \rho>0\eacc\subset A.$ 
\item[(a2)] There exists a point $x_0\in\rn$ such that $G_{\rn}^\alpha\rho(x_0)<\infty.$
\item[(a3)] There exists $\eta\geq0$ such that $\int_\eta^\infty r^{\alpha-1}\rho^*(r)\,dr<\infty$ where $\rho^*(r)=\sup_{\babsolu x\eabsolu=r}\rho(x).$
\eenu
In the particular case where $\rho$ is radially symmetric, (a2) and (a3) are necessary conditions as well.
\end{cor}
We conclude this paragraph by  answering the question araised in the begining of subsection 3.
\begin{pro}
The condition  (\ref{c11}) is not  necessary  for the existence of nontrivial entire bounded solution to Eq. (\ref{eip}).
\end{pro}
\begin{proof} Consider the thorn $A$ introduced in Example \ref{exemple1} and  take $\rho=1_A.$ Let $x\in \rn$ and choose $R>e$ and $c>0$ such that $G^\alpha_{\rn}(x,y)\geq c\,|y|^{\alpha-N}$ for all $y\in \rn$ satisfying $|y|\geq R$. Using  spherical coordinates in $\mathbb{R}^{N-1}$, we obtain 
\begin{eqnarray*}
G_{\rn}^\alpha\rho(x)=\int_{\rn}G^\alpha_{\rn}(x,y)\rho(y)\,dy &=& \int_AG^\alpha_{\rn}(x,y)\,dy\\
&\geq& c\int_{A\cap\{|y|\geq R\}}\frac{dy}{|y|^{N-\alpha}}\\
&=& c'\, \int_R^\infty\int_0^{h(r)}\frac{t^{N-2}}{(r^2+t^2)^{\frac{N-\alpha}{2}}}\,dt\,dr\\
&\geq& c'\,\int_R^\infty\left(\int_0^{h(r)}t^{N-2}\,dt\right)\frac{dr}{(r^2+h^2(r))^{\frac{N-\alpha}{2}}}\\
&=& \frac{c'}{N-1} \int_R^\infty\frac{(h(r))^{N-1}}{(r^2+h^2(r))^{\frac{N-\alpha}{2}}}\,dr\\
&=&\infty,
\end{eqnarray*}
since, for $r$ large enough, we have
$$
\frac{(h(r))^{N-1}}{(r^2+h^2(r))^{\frac{N-\alpha}{2}}}\simeq \frac{r^{\alpha-1}}{\left(\ln r\right)^{\beta(N-1)}}.
$$
Hence, the condition (\ref{c11}) fails while Eq. (\ref{eip}) admits a nonnegative nontrivial entire bounded solution by statement (a1) in corollary \ref{corollaire} since $\bacc \rho>0\eacc=A $  is $\alpha$-transient.
\end{proof}
\section{Nonpositive perturbation }
This last section is devoted to study the existence of positive entire  bounded solution (in the distributional sense) of Eq. (\ref{ein})
where $\varphi$ is a (nontrivial) nonnegative nonincreasing continuous function  in $]0,+\infty[$ and $\rho$ is a  nonnegative nontrivial function in $L_{\text{loc}}^\infty(\rn)$ (nontrivial in the sense that   the set $\bacc \rho>0\eacc$ has positive Lebesgue measure). First, let us point out that the comparison principle established in Lemma \ref{cprin} remains true in the nonpositive nonlinearity case. To be punctilious we rewrite it: \\
Let $D$ be a bounded open set and let $u,v\in C^+_b(\rn)$ such that \[ \daln u\geq \rho\,\varphi(u)\quad;\quad \daln v\leq \rho\,\varphi(v)\quad\text{in }D.\]
If $u\geq v$ in $D^c$ then $u\geq v$ in $\rn.$

We begin with the following result (in the linear case) which can be known.
\begin{pro}\label{linear}  Eq. $\daln u=\rho $ in $\rn$ admits an entire bounded solution if and only if 
\[\sup_{x\in\rn}\int_{\rn}\frac{\rho(y)}{\babsolu x-y\eabsolu^{N-\alpha}}dy<\infty.\]
\end{pro}
\begin{proof} The sufficiently is trivial. Let us proof the necessity. Suppose that $w$ is a bounded solution of the equation $\daln u=\rho $ in $\rn$. For every $k\geq1$ define $u_k=\int_{B_k}G_{B_k}^\alpha(.,y)\,\rho(y)\,dy$ which  is obviously the (unique)   solution of the problem
\[\bacc\ba{rcll}\daln u_k&=&\rho&\text{in }B_k\\
u_k&=&0&\text{in }B_k^c\ea\right.\]
 Note  that $(u_k)$ is a nondecreasing sequence of positive functions. Define $u=\sup_k u_k$ (possibly $\infty$). Recalling that $\sup_{k\geq1}G_{B_k}^\alpha=G_{\rn}^\alpha,$ we get that $u=G^\alpha_{\rn}\rho.$ On the other hand, by the comparison principle $u_k\leq w$ in $\rn$ for every $k\geq 1$ and therefore $u=G^\alpha_{\rn}\rho$ is bounded above by $w$ which in turn means that  $G^\alpha_{\rn}\rho$ is bounded in $\rn.$
\end{proof}
We would now solve   the following nonlinear Dirichlet problem
\begin{equation}\label{sldp}
\ba{rcll}\daln u&=&\rho\,\varphi(u)&\text{ in }D\\
u&=&f&\text{ in }D^c.\ea
\end{equation}
Let us before  review  some related results obtained for $\alpha=2.$ In \cite{delPino} the author considered (\ref{sldp}) in the case where $\varphi(t)=t^{-\gamma},$ $\gamma>0,$  $f\equiv 0$ on $\partial D$ and where $D$ is smooth. He proved the existence and uniqueness of the solution provided $\rho$ is nontrivial and bounded in $D.$ Later, this result was   extended  in \cite{8} to a more general function $\varphi$ and for any nonnegative continuous boundary datum $f.$
\begin{pro}\label{e-db}
Let $D$ be a bounded regular open set. For every $f\in C^+_b(D^c),$ the problem (\ref{sldp})
admits a unique solution $u\in C^+_b(\rn).$ Furthermore, for every $x\in \rn,$ we have
\[u(x)=H_D^\alpha f(x)+\int_DG_D^\alpha (x,y)\,\rho(y)\,\varphi(u(y))\,dy.\]
\end{pro}
\begin{proof} The uniqueness is a direct consequence of the comparison principle. In order to prove the existence, we suppose first that $f\geq c$ in $D^c$ where $c>0.$ In this case $h:=H_D^\alpha f\geq c$ in $D.$ We consider the following convex closed set
\[\Gamma:=\bacc u\in C^+_b(D)\ ;\ h\leq u\leq a\bnorme f\enorme\eacc,\]
where 
\[a:=\sup_{x\in D}E^x\bint \exp\bp \frac{1}{c}\varphi(c)\bp\sup_D \rho\ep\, \tau_D\ep\eint.\]
The existence of the constant $a$ is assured by the Gauge theorem \cite{chensonggauge}. Now define in $\Gamma $ the operator 
\[ Tu(x)=E^x\bint h(X_{\tau_D})\displaystyle{e^{\displaystyle{\int_0^{\tau_D}\rho(X_s)\frac{\varphi(u(X_s))}{u(X_s)}ds}}}\eint.\]
Let $u\in \Gamma.$ It is clear that $Tu\geq h.$ Besides, 
\[Tu(x)\leq \bnorme f\enorme E^x\bint e^{\frac{1}{c}\varphi(c)\bp\sup_D \rho\ep\, \tau_D}\eint\leq a\,\bnorme f\enorme;\quad\ x\in D\]
This yields that $T(\Gamma)\subset \Gamma.$ On the other hand, for every $u\in \Gamma,$ we have
\[\rho\frac{\varphi(u)}{u}Tu\leq \frac{1}{c}\,\varphi(c)\,a\,\bnorme f\enorme\,\sup_D\rho.\]
So, in virtue of Lemma \ref{l221}, we deduce that the family 
\[\bacc \int_D G_D^\alpha(.,y)\rho(y)\,\frac{\varphi(u(y))}{u(y)}Tu(y)\,dy\quad;\quad u\in \Gamma\eacc\]
is relatively compact in $C^+(D).$ The well-known Feynman-Kac theorem implies that  
\[Tu(x)=h(x)+\int_D G_D^\alpha(x,y)\rho(y)\,\frac{\varphi(u(y))}{u(y)}Tu(y)\,dy\quad;\quad x\in D.\]
It follows that $T(\Gamma)$ is relatively compact in $C^+(D).$ By Schauder's fixed point theorem, there exists $u\in \Gamma$ such that 
\[u(x)=h(x)+\int_D G_D^\alpha (x,y)\,\rho(y)\,\varphi(u(y))\, dy\quad;\quad x\in D.\]
Since $\rho\varphi(u)$ is bounded in $D,$ we have $G_D^\alpha(\rho\varphi(u))\in \mathcal{C}_0(D)$ and consequently $u=f$ in $D^c.$ By (a straightforward modification of ) Lemma \ref{l1}, we deduce that $u$ is a solution to problem (\ref{sldp}).

We now return to the general case where $f$ is an arbitrary nonnegative bounded continuous function in $D^c.$

For every $k\geq 1,$ let $u_k$ be the (unique) positive solution of problem (\ref{sldp}) for $f_k=f+\frac{1}{k}.$ Then the sequence $(u_k),$ by statement (a) in Lemma (\ref{pcg}), is  nonincreasing and by the first step is  in $\cc{C}^+_b(\rn)$ and satisfying 
\[u_k(x)=H_D^\alpha f_k(x)+\int_D G_D^\alpha(x,y)\,\rho(y)\,\varphi(u_k(y))\,dy\quad;\quad x\in \rn,\]
for every $k\geq 1.$ Define $u=\inf_{k\geq 1}u_k.$ Letting $k$ tend to $\infty$ we obtain that $G_D^\alpha(\rho\varphi(u))<\infty$ and 
\[u(x)=H_D^\alpha f(x)+\int_D G_D^\alpha(x,y)\,\rho(y)\,\varphi(u(y))\,dy\quad;\quad x\in \rn.\]
Note that $\rho\varphi(u)$ is eventually unbounded and so $G_D^\alpha(\rho\varphi(u))$ is not necessary zero on $\partial D.$ However, for $z\in\partial D$ we can see that $f(z)=\liminf_{x\to z}H^\alpha_D f(x)\leq \liminf_{x\to z}u(x).$  Besides, $\limsup_{x\to z}u(x)\leq \limsup_{x\to z}u_k(x)=f(z)+\frac{1}{k}$ for every $k\geq 1$ and so $\limsup_{x\to z}u(x)\leq f(z).$ We deduce that $\lim_{x\to z}u(x)=f(z)$ for every $z\in\partial D$ ( in other words,  $u=f$ on $D^c$), concluding the proof.
\end{proof}
We are now in position to state the main theorem of this section.
\begin{thm}
The Eq. (\ref{ein}) admits a nonnegative bounded entire solution if and only if 
\[\sup_{x\in\rn}\int_{\rn}\frac{\rho(y)}{\babsolu x-y\eabsolu^{N-\alpha}}dy<\infty.\]
\end{thm}
\begin{proof} First we prove the sufficiently. For every $k\geq1,$ let $u_k$ be the nonnegative bounded solution of the semilinear Dirichlet problem
\[\ba{rcll}\daln u_k&=&\rho\,\varphi(u_k)&\text{ in }B_k\\
u_k&=&0&\text{ in }B_k^c\ea\]
where $B_k$ denotes the ball $B(0,k).$ Such a solution exists in virtue of  proposition \ref{e-db}. By the comparison principle, $(u_k)$ is nondecreasing. Define $u=\sup_k u_k.$ We claim that $u$ is bounded and this  achieves the if part. Consider the nonnegative bounded function $w:=G^\alpha_{\rn} \rho,$ solution of the equation $\daln=\rho.$  
Define 
\begin{equation}\label{psi}\Psi:[0,\infty[\avaleur [0,\infty[,\quad t\associe \int_0^t\frac{ds}{\varphi(s)}.\end{equation}
 It is obvious that $\Psi$  is continuous and increasing and therefore it is invertible from  $[0,\infty[$ to $[0,\infty[ $ ( $\lim_{t\to\infty}\Psi(t)=\infty$ since $\frac{1}{\varphi}$ is increasing ).  
 We consider then the function $v$ defined in $[0,\infty[$ by $v=\Psi^{-1}(w).$ Note that $v$ is bounded since $\Psi^{-1}$ is continuous in $[0,\infty[$ and $w$ is bounded. Seeing that $\Psi$ is convex, we obtain
\begin{displaymath}
\begin{split}
-\rho=\dal w&= \int_{\rn}\frac{\Psi(v(x))-\Psi(v(y))}{\babsolu y-x\eabsolu^{N+\alpha}}\,dy\\
&\geq \int_{\rn}\frac{v(x)-v(y)}{\babsolu y-x\eabsolu^{N+\alpha}}\Psi'(v(x))\,dy\\
&=\dal v(x)\frac{1}{\varphi(v(x))}.
\end{split}
\end{displaymath} 
Therefore, $\daln v\geq \rho\varphi(v)$ in $\rn.$ By the comparison principle, $u_k\leq v$ for every $k\geq 1,$  so   that $u\leq v$ and this means that $u$ is bounded  as desired. 

To prove the necessity, let $u$ be a nontrivial entire bounded  solution to  Eq. (\ref{ein}). Then, by proposition \ref{e-db}, for every $k\geq1$, we have 
\[u(x)=H_{B_k}^\alpha u(x) +\int_{B_k} G_{B_k}^\alpha(x,y)\rho(y)\,\varphi(u(y))\,dy\quad ;\quad x\in \rn.\]
Consequently, for every $k\geq1$
\[\varphi(\bnorme u\enorme)\int_{B_k} G_{B_k}^\alpha(x,y)\rho(y)\,dy\leq \int_{B_k} G_{B_k}^\alpha(x,y)\rho(y)\,\varphi(u(y))\,dy\leq u(x)\quad;\quad x\in\rn.\]
Letting $k$ tend to $\infty,$ by monotone convergence theorem, we obtain
\begin{equation}\label{tendverszero}G^\alpha_{\rn}\rho=\int_{\rn} G_{\rn}^\alpha(\cdot,y)\rho(y)\,dy\leq \frac{u}{\varphi(\bnorme u\enorme)}\leq \frac{\bnorme u\enorme}{\varphi(\bnorme u\enorme)},
\end{equation}
and the necessity is proved.
\end{proof}
\begin{pro}\label{comportement} Eq. (\ref{ein}) has a nonnegative bounded solution decaying to zero at infinity if and only if
\begin{equation}\label{final}\lim_{\babsolu x\eabsolu\to \infty}\int_{\rn}\frac{\rho(y)}{\babsolu x-y\eabsolu^{N-\alpha}}dy=0.
\end{equation}
\end{pro}
\begin{proof} This proof is a furtherance of the previous one. If $u$ is a bounded solution of Eq. (\ref{ein}) vanishing at infinity then by (\ref{tendverszero}), $G^\alpha_{\rn}\rho\leq \frac{u}{\varphi(\bnorme u\enorme)}$ and so (\ref{final}) holds true.   Conversely,  by the previous proof,  there exists a  nonnegative bounded solution $u$ of Eq. (\ref{ein}) such that $u\leq \Psi^{-1}(G^\alpha_{\rn}\rho)$ where $\Psi$ is given by (\ref{psi}). Then, the hypothesis (\ref{final}) implies clearily that the  solution $u$ tends  to zero at infinity since $\Psi^{-1}$ is continous and $\Psi^{-1}(0)=0.$
\end{proof}
\begin{cor} Assume that $\rho$ is radially symmetric. Then the following statements are equivalent.
\benu
\item[(a1)] There exists $\eta>0$ such that $\int_\eta^\infty r^{\alpha-1}\rho(r)\,dr<\infty.$
\item[(a2)] Eq. (\ref{ein}) has a positive entire solution in $\rn$ decaying to zero at infinity.
\item[(a3)] Eq. (\ref{ein}) has a positive entire solution in $\rn.$
\eenu
\end{cor}
\begin{proof} It follows from (a1) that $G_{\rn}^\alpha\rho(0)<\infty$ and consequently $G_{\rn}^\alpha\rho$ is a potential in $\rn.$ Hence, by \cite[Theorem 2]{mizuta}, $\lim_{\babsolu x\eabsolu\to\infty}G_{\rn}^\alpha\rho(x)=0$ and so (a2) is obtained by proposition \ref{comportement}. Now it is obvious that (a2) implies (a3). Finally, if we suppose that (a3) holds then $\sup_{x\in\rn}G_{\rn}^\alpha\rho(x)<\infty$ and in particular $G_{\rn}^\alpha\rho(0)=\int_0^\infty r^{\alpha-1}\rho(r)\,dr<\infty.$ Thus (a3) implies (a1).
\end{proof}

\end{document}